\documentclass[12pt]{amsart}

\usepackage[headings]{fullpage}
\usepackage{amsmath}
\usepackage{amssymb}

\newtheorem{theorem}{Theorem}
\newtheorem*{theorem*}{Theorem}
\newtheorem{thmalpha}{Theorem}
\newtheorem{lemma}{Lemma}
\newtheorem{proposition}{Proposition}
\newtheorem*{proposition*}{Proposition}
\newtheorem{corollary}{Corollary}
\theoremstyle{definition}
\newtheorem{definition}{Definition}
\theoremstyle{remark}
\newtheorem{remark}{Remark}
\newtheorem*{remark*}{Remark}
\newtheorem*{remarkn*}{Remark on Notation}

\newtheorem*{note*}{Note}

\begin{document}

\title[Positive Matrices via the Kaczmarz Algorithm]{Positive Matrices in the Hardy Space with Prescribed Boundary Representations via the Kaczmarz Algorithm}
\author{John E. Herr}
\address[John E. Herr and Eric S. Weber]{Department of Mathematics, Iowa State University, 396 Carver Hall, Ames, IA 50011}
\email{jherr@iastate.edu, esweber@iastate.edu}
\author{Palle E. T. Jorgensen}
\address[Palle E.T. Jorgensen]{Department of Mathematics, University of Iowa, Iowa City, IA 52242}
\email{palle-jorgensen@uiowa.edu}
\author{Eric S. Weber}
\date{\today}
\begin{abstract}
For a singular probability measure $\mu$ on the circle, we show the existence of positive matrices on the unit disc which admit a boundary representation on the unit circle with respect to $\mu$.  These positive matrices are constructed in several different ways using the Kaczmarz algorithm.  Some of these positive matrices correspond to the projection of the Szeg\H{o} kernel on the disc to certain subspaces of the Hardy space corresponding to the normalized Cauchy transform of $\mu$.  Other positive matrices are obtained which correspond to subspaces of the Hardy space after a renormalization, and so are not projections of the Szeg\H{o} kernel.  We show that these positive matrices are a generalization of a spectrum or Fourier frame for $\mu$, and the existence of such a positive matrix does not require $\mu$ to be spectral.
\end{abstract}
\maketitle

\section{Introduction}

The boundary value problems we consider here are motivated by two cases considered earlier. One is Fatou's theorem for the Hardy space $H^2$ on the disk $\mathbb{D}$, yielding an isomorphism between $H^2$ on the one hand, and $L^2$ of the boundary circle $\mathbb{T}$ on the other, with the $L^2$ referring to the Haar (normalized Lebesgue) measure on $\mathbb{T}$. In particular, this theorem shows that every $f$ in $H^2$ has a non-tangential limit a.e. with respect to Lebesgue measure on $\mathbb{T}$, and that the $L^2$ norm of the limit function agrees with the $H^2$ norm of $f$. Now because of a more general duality theory, it is natural to ask for boundary representations specified by certain lacunary subspaces in $H^2$. On account of work by the second named author and co-authors (e.g. \cite{JP98,DJ11}, see also \cite{Str00}), it is known that there are families of singular measures $\mu$ on the circle that admit a Fourier duality corresponding to associated sets of lacunary Fourier frequencies. (Such measures $\mu$ are said to be spectral.)  The latter sets of frequencies in turn index certain closed subspaces of $H^2$ that can be shown to have boundary representations, referring now instead to $L^2(\mu)$ boundary values. By ``lacunary'' we refer to Fourier series having asymptotically an infinite sequence of gaps between non-zero coefficients, the successive gaps growing at a geometric rate.

The boundary representations we present here go beyond that of spectral measures from spectral pairs.  Our results in turn are based on a new kernel analysis, and we open below by recalling some facts from the theory of reproducing kernels that will be used inside the paper. We make use of frames and of the structure theorem of Wold for isometries in Hilbert spaces. Our frame expansions are constructive in that we generate them from the Kaczmarz algorithm, a procedure originally used to solve systems of linear equations.

The classical Hardy space $H^2$ consists of those holomorphic functions $f$ defined on $\mathbb{D}$ satisfying
\begin{equation}\label{Hardy}\lVert f\rVert^2_{H^2}:=\sup_{0<r<1}\int_{0}^{1}\lvert f(re^{2\pi ix})\rvert^2\,dx<\infty.\end{equation}
It is well-known that an equivalent description of $H^2$ is as the space of holomorphic functions on $\mathbb{D}$ with square-summable coefficients:
$$H^2=\left\{\sum_{n=0}^{\infty}c_nz^n~\middle|~\sum_{n=0}^{\infty}\lvert c_n\rvert^2<\infty\right\},$$
where the norm is then equivalently given by
$$\lVert f\rVert^2_{H^2}=\sum_{n=0}^{\infty}\lvert c_n\rvert^2.$$
In addition, for each $f\in H^2$, there exists a (unique) function $f^\ast\in L^2(\mathbb{T})$, which we shall call the Lebesgue boundary function of $f$, such that
\begin{equation}\label{LebBoundary}\lim_{r\rightarrow1^{-}}\int_{0}^{1}\lvert f(re^{2\pi ix})-f^\ast(e^{2\pi ix})\rvert^2\,dx=0.\end{equation}
In fact, $\lim_{r\rightarrow1^-}f(re^{2\pi i x})=f^\ast(e^{2\pi ix})$ pointwise for almost every $x$. If $f(z)=\sum_{n=0}^{\infty}a_n z^n$ and $g(z)=\sum_{n=0}^{\infty}b_n z^n$ are two members of $H^2$, the inner product of $f$ and $g$ in $H^2$ can be described in two ways:
\begin{equation*}
{\langle f,g\rangle}_{H^2}=\sum_{n=0}^{\infty}a_n\overline{b_n}=\int_{0}^{1}f^\ast(e^{2\pi ix})\overline{g^\ast(e^{2\pi ix})}\,dx.\end{equation*}
Because the point-evaluation functionals on the Hardy space are bounded, the Hardy space is a reproducing kernel Hilbert space. Its kernel is the classical Szeg\H{o} kernel $k(z,w)=:k_z$, defined by
$$k_z(w):=\frac{1}{1-\overline{z}w}.$$
We then have
\begin{equation*}f(z)=\langle f,k_z\rangle_{H^2}=\int_{0}^{1}f^\ast(e^{2\pi ix})\overline{k_z^\ast(e^{2\pi ix})}\,dx\end{equation*}
for all $f\in H^2$. In particular,
\begin{equation}\label{szegorep}k_z(w):=\int_{0}^{1}k_z^\ast(e^{2\pi ix})\overline{k_w^\ast(e^{2\pi ix})}\,dx.\end{equation} Equation $\eqref{szegorep}$ shows that the Szeg\H{o} kernel reproduces itself with respect to what is, by some definition, its boundary. The measure on the circle used to define $k_z^\ast$ in \eqref{LebBoundary} is Lebesgue measure, as is the measure in $\eqref{szegorep}$. The intent of this paper is to show that among the functions in the Hardy space, there are a host of other kernels that reproduce with respect to their boundaries. However, these boundary functions will not be taken with respect to Lebesgue measure, but with respect to a given singular measure, and the integration of these boundary functions will also be done with respect to this measure. We consider two main questions: which positive matrices does the Hardy space contain that reproduce themselves by boundary functions with respect to a given measure, and with respect to which measures will a positive matrix reproduce itself by boundary functions?

\begin{remark}In this paper, we will be dealing with measures $\mu$ on the unit circle. The unit circle $\mathbb{T}:=\{z\in\mathbb{C}:\lvert z\rvert=1\}$ and its topology shall be identified with $[0,1)$ via the relation $\xi=e^{2\pi ix}$ for $\xi\in\mathbb{T}$ and $x\in[0,1)$. We will regard the measures $\mu$ as being supported on $[0,1)$. A function $f(\xi)$ defined on $\mathbb{T}$ (for example, a boundary function) may be regarded as being in $L^2(\mu)$ if $f(e^{2\pi ix})\in L^2(\mu)$. For aesthetics, the inner product (norm) in $L^2(\mu)$ will be denoted $\langle\cdot,\cdot\rangle_{\mu}$ ($\lVert\cdot\rVert_{\mu}$) rather than $\langle\cdot,\cdot\rangle_{L^2(\mu)}$ ($\lVert\cdot\rVert_{L^2(\mu)}$). The subscript will be suppressed where context suffices.  A measure $\mu$ will be called \emph{singular} if it is a Borel measure that is singular with respect to Lebesgue measure.
\end{remark}

\begin{definition}If $\mu$ is a finite Borel measure on $[0,1)$ and $f(z)$ is an analytic function on $\mathbb{D}$, we say that $f^\ast\in L^2(\mu)$ is an $L^2(\mu)$-boundary function of $f$ if
\begin{equation*}\lim_{r\rightarrow1^-}\left\lVert f^\ast (x)-f(re^{2\pi ix})\right\rVert_{\mu}=0.\end{equation*}
If a function possesses an $L^2(\mu)$-boundary function, then clearly that boundary function is unique. The $L^2(\mu)$-boundary function of a function $f:\mathbb{D}\rightarrow\mathbb{C}$ shall be denoted $f^\ast_\mu$, but we omit the subscript when context precludes ambiguity.
\end{definition}

\begin{definition}A positive matrix (in the sense of E.~H.~Moore) on a domain $E$ is a function $K(z,w):E\times E\rightarrow\mathbb{C}$ such that for all finite sequences $\zeta_1,\zeta_2,\ldots,\zeta_n\in E$, the matrix
$$(K(\zeta_j,\zeta_i))_{ij}$$
is positive semidefinite. We will usually write $K_z(w)$ instead of $K(z,w)$, to emphasize that each fixed $z$ yields a function in $w$. Given a positive matrix $K_z(w)$, we will use the bare notation $K$ to refer to the set $\{K_z:z\in E\}$ of functions from $E$ to $\mathbb{C}$ comprising it, though sometimes we will use $K$ to refer to the positive matrix itself as a function from $E\times E$ to $\mathbb{C}$. 
\end{definition}

Our interest is in positive matrices on $E=\mathbb{D}$, and more specifically those residing in $H^2(\mathbb{D})$. Recall that the classical Hardy space is a reproducing kernel Hilbert space. We therefore desire to find subspaces of the Hardy space that not only are Hilbert spaces with respect to the $L^2(\mu)$-boundary norm, but are in fact reproducing kernel Hilbert spaces with respect to this norm. The classical Moore-Aronszajn Theorem connects positive matrices to reproducing kernel Hilbert spaces \cite{Aron50a}:

\begin{theorem*}[Moore-Aronszajn]To every positive matrix $K_z(w)$ on a domain $E$ there corresponds one and only one class of functions on $E$ with a uniquely determined quadratic form in it, forming a Hilbert space and admitting $K_z(w)$ as a reproducing kernel. This class of functions is generated by all functions of the form $\sum_{k=1}^{n}\xi_k K_{z_k}(w)$, with norm given by
$$\left\lVert\sum_{k=1}^{n}\xi_k K_{z_k}(w)\right\rVert^2=\sum_{i,j=1}^{n}K_{z_j}(z_i)\overline{\xi_i}\xi_j.$$
\end{theorem*}

Conversely, every reproducing kernel of a Hilbert space of functions on a common domain is a positive matrix. Let us then define two sets of interest:

\begin{definition}Let $\mu$ be a nonnegative Borel measure on $[0,1)$. We define $\mathcal{K}(\mu)$ to be the set of positive matrices $K$ on $\mathbb{D}$ such that for each fixed $z\in\mathbb{D}$, $K_z$ possesses an $L^2(\mu)$-boundary $K_z^\ast$, and $K_z(w)$ reproduces itself with respect to integration of these $L^2(\mu)$-boundaries, i.e.
$$K_z(w)=\int_{0}^{1}K_z^\ast(x)\overline{K_w^\ast(x)}\,d\mu(x)$$ for all $z,w\in\mathbb{D}$.\end{definition}

\begin{definition}Let $K$ be a positive matrix on $\mathbb{D}$. We define $\mathcal{M}(K)$ to be the set of nonnegative Borel measures $\mu$ on $[0,1)$ such that for each fixed $z\in\mathbb{D}$, $K_z$ possesses an $L^2(\mu)$-boundary $K_z^\ast$, and $K_z(w)$ reproduces itself with respect to integration of these $L^2(\mu)$-boundaries.\end{definition}

\begin{definition}A sequence $\{x_n\}_{n=0}^{\infty}$ in a Hilbert space $\mathbb{H}$ is called a frame \cite{DS52} if there exist positive constants $A$ and $B$ such that
\begin{equation}\label{framecond}A\lVert\phi\rVert^2\leq\sum_{n=0}^{\infty}\lvert\langle \phi,x_n\rangle\rvert^2\leq B\lVert\phi\rVert^2\end{equation}
for all $\phi\in\mathbb{H}$. If $\{x_n\}_{n=0}^{\infty}$ satisfies (possibly only) the right-hand inequality in \eqref{framecond}, it is called a Bessel sequence. If $A=B$, the frame is called tight, and if $A=B=1$, it is called a Parseval frame.\end{definition}

\begin{definition}Let $\{x_n\}_{n=0}^{\infty}$ be a frame in a Hilbert space $\mathbb{H}$. A frame $\{y_n\}_{n=0}^{\infty}$ in $\mathbb{H}$ is a dual frame of $\{x_n\}_{n=0}^{\infty}$ if \begin{equation}\label{dualframecond}\sum_{n=0}^{\infty}\langle \phi,x_n\rangle y_n=\phi\text{ for all }\phi\in\mathbb{H}.\end{equation} If \eqref{dualframecond} is satisfied, then it is necessarily also true that \begin{equation*}\sum_{n=0}^{\infty}\langle\phi,y_n\rangle x_n=\phi\text{ for all }\phi\in\mathbb{H}.\end{equation*} Thus, frame duality is symmetric. A given frame will generally have many dual frames, but every frame possesses a unique canonical dual frame. A Parseval frame is its own canonical dual.\end{definition}

The quaternary Cantor measure $\mu_4$ is the restriction of the $\frac{1}{2}$-dimensional Hausdorff measure to the quaternary Cantor set. Likewise, the ternary Cantor measure $\mu_3$ is the restriction of the $\frac{\ln(2)}{\ln(3)}$-dimensional Hausdorff measure to the ternary Cantor set. In \cite{JP98}, Jorgensen and Pedersen showed that the quaternary Cantor measure is spectral. That is, there exists a set $\Gamma \subset\mathbb{Z}$ such that the set of complex exponentials $\left\{e^{2\pi i\lambda x}\right\}_{\lambda \in \Gamma}$ is an orthonormal basis of $L^2(\mu_4)$. From this, Dutkay and Jorgensen \cite{DJ11} constructed a positive matrix $G_\Gamma$ inside $H^2$ that reproduces itself both in $H^2$ and with respect to $L^2(\mu_4)$-boundary integration. Thus $G_\Gamma\in\mathcal{K}(\mu_4)$.

In \cite{JP98}, it was also shown that $\mu_3$ is not spectral. Thus, it is not possible to construct a positive matrix for $\mu_3$ in the same way as for $\mu_{4}$. However, it is sufficient for $\mu_{3}$ to possess an exponential frame:
\begin{proposition*}
If there exists a sequence $\{n_{j}\}_{j=0}^{\infty}$ of nonnegative integers such that $\{ e^{2 \pi i n_{j} x} : j \geq 0\} \subset L^2(\mu)$ is a frame, then $\mathcal{K}(\mu)$ is nonempty.
\end{proposition*}
We will prove a generalization of this result in Theorem \ref{kernelthm}.  It is still unknown whether $\mu_3$ possesses an exponential frame. Despite this seeming impediment, in this paper we will show not only that $\mathcal{K}(\mu_3)$ is nonempty, but that it contains infinitely many members within $H^2$. In fact, we will show this for all singular probability measures on $[0,1)$.

\section{Kernels in $\mathcal{K}(\mu)$ That Are Also $H^2$ Kernels}
In our first step to show that for $\mu$ a singular probability measure $\mathcal{K}(\mu)$ has a rich variety of inhabitants, we consider when projections of the Szeg\H{o} kernel onto appropriate subspaces of $H^2$ will be elements in $\mathcal{K}(\mu)$.  For the measure $\mu$, there is a canonical subspace of $H^2$ identified with $\mu$--it is the image of $L^2(\mu)$ under the Normalized Cauchy transform, which also is a de Branges-Rovnyak space.  This subspace will give rise to many kernels in $\mathcal{K}(\mu)$.

\begin{definition}A function $b\in H^\infty(\mathbb{D})$ (the space of bounded holomorphic functions on $\mathbb{D}$) is said to be inner if the radial limits $b^\ast(e^{2\pi ix}):=\lim_{r\rightarrow 1^-}b(re^{2\pi ix})$ exist for almost all $x\in[0,1)$ with respect to Lebesgue measure and $\lvert b^\ast(e^{2\pi ix})\rvert=1$ for almost all $x$.\end{definition}

There is a one-to-one correspondence between the nonconstant inner functions $b$ and the finite nonnegative singular measures $\mu$ on $[0,1)$ given by the Herglotz representation:
\begin{equation}\label{innerpoisson}\text{Re}\left(\frac{1+b(z)}{1-b(z)}\right)=\int_{0}^{1}\frac{1-\lvert z\rvert^2}{\lvert e^{2\pi ix}-z\rvert^2}\,d\mu(x).\end{equation}
For a singular measure $\mu$ and an inner function $b$ related in this way, we will say that $\mu$ is the measure ``corresponding'' to $b$, or that $b$ is the inner function ``corresponding'' to $\mu$.

Let $S$ denote the forward shift on $H^2$, i.e. $Sf(z) = zf(z)$. Beurling's Theorem says that the nontrivial invariant subspaces for $S$ are those subspaces of $H^2$ of the form $bH^2$, where $b$ is an inner function. The nontrivial invariant subspaces of the backward shift $S^\ast$ $\left(S^\ast f(z) = \frac{ f(z) - f(0) }{z} \right)$ are then of the form $H^2\ominus bH^2$, where $b$ is inner. For each $b\in H^\infty$, there is a de Branges-Rovnyak space $\mathcal{H}(b)$ \cite{dBR66a,ADV09a}, which is by definition the range of the operator $A=(I-T_bT_{\overline{b}})^{1/2}:H^2\rightarrow H^2$ along with the Hilbert space structure that makes $A$ a partial isometry from $H^2$ to $\mathcal{H}(b)$. (Here $T_b$ is the Toeplitz operator with symbol $b$.) In this paper, we are only concerned with the situation in which $b$ is inner, and in that case we have $\mathcal{H}(b)=H^2\ominus bH^2$ with the norm inherited from $H^2$. For a complete treatment, see Sarason's book \cite{Sar94}.

\begin{definition}For a finite nonnegative Borel measure $\mu$ on $[0,1)$, we define the normalized Cauchy transform $V_\mu$ from $L^1(\mu)$ to the set of functions on $\mathbb{C}\setminus\mathbb{T}$ by
\begin{equation}\label{NCtransform}V_\mu f(z)=\frac{\displaystyle\int_{0}^{1}\frac{f(x)}{1-ze^{-2\pi ix}}\,d\mu(x)}{\displaystyle\int_{0}^{1}\frac{1}{1-ze^{-2\pi ix}}\,d\mu(x)}.\end{equation}\end{definition}
If $\mu$ is a singular probability measure on $[0,1)$ with corresponding inner function $b$, then $V_\mu$ is an isometry of $L^2(\mu)$ onto $\mathcal{H}(b)$ \cite{Clark72,Sar90,Pol93,HW15}.

The unnormalized Cauchy transform shall here be denoted $C_\mu$ (in \cite{Sar94}, $K_\mu$):
\begin{equation}\label{UCtransform}C_\mu f(z)=\int_{0}^{1}\frac{f(x)}{1-ze^{-2\pi ix}}\,d\mu(x).\end{equation}

Define $e_\lambda(x):=e^{2\pi i\lambda x}$. In \cite{HW15}, it was proved that if $\mu$ is a singular probability measure, then the sequence $\{g_n\}_{n=0}^{\infty}$ defined by
\begin{equation}
\label{gs}g_0=e_0, \qquad
g_n=e_n-\sum_{i=0}^{n-1}\langle e_n,e_i\rangle g_i. 
\end{equation}
is a Parseval frame in $L^2(\mu)$ satisfying
\begin{equation}\label{gndextro}\sum_{n=0}^{\infty}\langle f,g_n\rangle e_n=f\end{equation}
in norm for all $f\in L^2(\mu)$. Equations (\ref{gs}) and (\ref{gndextro}) are referred to as the Kaczmarz algorithm \cite{Kacz37,KwMy01}.  Equation (\ref{gndextro}) can be interpreted as a Fourier expansion of $f \in L^2(\mu)$; see also \cite{Pol93,Str06a}.

There exists a sequence $\{\alpha_n\}$ of scalars (depending on $\mu$) such that
\begin{equation}\label{gnformula}g_n=\sum_{i=0}^{n}\overline{\alpha_{n-i}}e_i\end{equation}
for all $n\in\mathbb{N}_0$.  This sequence is obtained by inverting a lower triangular banded matrix whose $jk$-th entry is $\hat{\mu}(j-k)$.  For an explicit expression, see \cite{HW15}. The following was also proved:

\begin{thmalpha}[\cite{HW15}]\label{PapOneMain} If $\mu$ is a singular probability measure, then for all $f\in L^2(\mu)$,
\begin{equation}\label{NCTalt}V_\mu f(z)=\sum_{n=0}^{\infty}\langle f,g_n\rangle z^n.\end{equation}
\end{thmalpha}


The following is proven in \cite{Pol93}; we give an alternate proof here using Theorem \ref{PapOneMain}.

\begin{proposition}\label{muboundprop}If $\mu$ is a singular probability measure and $f\in L^2(\mu)$, then $f$ is an $L^2(\mu)$-boundary function of $V_\mu f(z)$. Consequently, for any $F\in\mathcal{H}(b)$, $V_\mu^{-1}F=F^\ast$.\end{proposition}

\begin{proof}
Since the sum in \eqref{gndextro} is summable in $L^2(\mu)$, it is Abel summable, and hence by \eqref{NCTalt} we have that
\begin{equation*}\lim_{r\rightarrow1^-}V_\mu f(re^{2\pi ix})=\lim_{r\rightarrow1^-}\sum_{n=0}^{\infty}\langle f,g_n\rangle r^ne_n=\sum_{n=0}^{\infty}\langle f,g_n\rangle e_n=f\end{equation*}
in the $L^2(\mu)$ norm. Hence, $f$ is an $L^2(\mu)$-boundary function of $V_\mu f(z)$.

Now if $F\in\mathcal{H}(b)$, then by bijectivity of $V_\mu$, there exists a  $f\in L^2(\mu)$ such that $V_\mu f(z)=F(z)$. Then $f$ is an $L^2(\mu)$-boundary of $V_\mu f(z)=F(z)$, and since an $L^2(\mu)$-boundary is unique, we have $F^\ast=f$. Hence, $V_\mu^{-1}F=F^\ast$.
\end{proof}

\begin{corollary}\label{dBRbound} If $\mu$ is a singular probability measure with corresponding inner function $b$, then for any $f(z),j(z)\in\mathcal{H}(b)$, we have
\begin{equation}{\langle f,j\rangle}_{\mathcal{H}(b)}={\langle f^\ast,j^\ast\rangle}_{\mu},\end{equation}
where $f^\ast$ and $j^\ast$ are the $L^2(\mu)$-boundary functions of $f$ and $j$, respectively.\end{corollary}

\begin{proof}Since $V_\mu$ is an isometry from $L^2(\mu)$ to $\mathcal{H}(b)$, Proposition \ref{muboundprop} implies
$$\langle f,j\rangle_{\mathcal{H}(b)}=\langle V_{\mu}^{-1}f,V_{\mu}^{-1}j\rangle_{\mu}=\langle f^\ast,j^\ast\rangle_{\mu}.$$
\end{proof}

Thus, for inner functions $b$ with $b(0) = 0$, functions in $\mathcal{H}(b)$ not only have Lebesgue boundaries, but also $L^2(\mu)$-boundaries, and the norm of $\mathcal{H}(b)$ is equal to boundary integration with respect to either boundary/measure pair. As an ordinary subspace of $H^2$, $\mathcal{H}(b)$ is of course a reproducing kernel Hilbert space. Let $k_z(w)\in H^2$ denote the Szeg\H{o} kernel of $H^2$. 
It is known (see \cite{Sar94}) that the kernel of $\mathcal{H}(b)$ is given by
\begin{equation*}k^b_z(w)=(1-\overline{b(z)}b(w))k_z(w).\end{equation*}
Using \eqref{NCTalt}, we give the following alternative form:

\begin{theorem}\label{DRkernel}Let $\mu$ be a singular probability measure with corresponding inner function $b$ and associated sequence $\{g_n\}_{n=0}^{\infty}\subset L^2(\mu)$ defined by \eqref{gs}. Then
\begin{equation} \label{Eq:DRkernel}
k_z^b(w)=\sum_{m=0}^{\infty}\sum_{n=0}^{\infty}\langle g_n,g_m\rangle_{\mu}\overline{z}^n w^m.
\end{equation}
\end{theorem}

\begin{proof}  We can combine \eqref{innerpoisson} with a result in \cite{HW15} (which uses ideas in \cite{KwMy01}) to obtain that the inner function $b$ satisfies
\begin{equation}  \label{InnerByUCtransform} 
b(z)=1-\frac{1}{C_\mu1(z)} = 1 - \sum_{n=0}^{\infty}\alpha_{n}z^n = -\sum_{n=1}^{\infty}\alpha_n z^n.
\end{equation}
%
Since the sequence $\{g_n\}_{n=0}^{\infty}$ is Bessel, $\sum_{n=0}^{\infty}\overline{z}^n g_n$ converges in $L^2(\mu)$ for all $z\in\mathbb{D}$. Observe that for a fixed $z\in\mathbb{D}$,
\begin{align}\sum_{n=0}^{\infty}\overline{z}^ng_n&=\sum_{n=0}^{\infty}\overline{z}^n\left(\sum_{j=0}^{n}\overline{\alpha_{n-j}}e_j\right) \notag \\
&=\sum_{j=0}^{\infty}\sum_{n=0}^{\infty}\overline{z}^{n+j}\overline{\alpha_n}e_j \notag \\
&=\left(\sum_{n=0}^{\infty}\overline{\alpha_{n}}\overline{z}^n\right)\left(\sum_{j=0}^{\infty}\overline{z}^je_j\right) \notag \\
&=(1-\overline{b(z)})k_z^\ast. \label{Eq:zngn}
\end{align}
The rearrangement of summation above is justified, because
\[ \sum_{j=0}^{\infty}\sum_{n=0}^{\infty}\lVert\overline{z}^{n+j}\overline{\alpha_n}e_j\rVert \leq \sum_{j=0}^{\infty}\lvert z\rvert^j\sqrt{\sum_{n=0}^{\infty}\lvert z^2\rvert^n}\sqrt{\sum_{n=0}^{\infty}\lvert\alpha_n\rvert^2} < \infty \]
which shows that the sum converges absolutely.
Recall from Theorem \ref{PapOneMain} that for $f\in L^2(\mu)$, $V_\mu f(w)=\sum_{n=0}^{\infty}\langle f,g_n\rangle w^n$. Therefore, we have
\begin{align*}V_\mu\left[\sum_{n=0}^{\infty}\overline{z}^ng_n\right](w)&=\sum_{m=0}^{\infty}\left\langle\sum_{n=0}^{\infty}\overline{z}^ng_n,g_m\right\rangle w^m\\
&=\sum_{m=0}^{\infty}\sum_{n=0}^{\infty}\langle g_n,g_m\rangle\overline{z}^nw^m.\end{align*}

On the other hand, in \cite{Sar94} it is computed via the Herglotz representation that
\[ 
C_{\mu}k_z^\ast(w) = (1-\overline{b(z)})^{-1}(1-b(w))^{-1}k_z^b(w).
\]
Therefore (by $V_\mu$'s original definition, but in accordance with \eqref{NCtransform}, \eqref{UCtransform}, and \eqref{InnerByUCtransform}),
\begin{align*}V_{\mu}\left[(1-\overline{b(z)})k_z^\ast\right](w)&:=(1-b(w))C_{\mu}\left[(1-\overline{b(z)})k_z^\ast\right](w)\\
&=(1-b(w))(1-\overline{b(z)})C_{\mu}k_z(w)\\
&=k_z^b(w).
\end{align*}
Equation (\ref{Eq:DRkernel}) now follows from Equation (\ref{Eq:zngn}).
\end{proof}


\begin{theorem}\label{dBRkern}If $\mu$ is a singular probability measure on $[0,1)$ with corresponding inner function $b$, then $k^b\in\mathcal{K}(\mu)$, and $\mu\in\mathcal{M}(k_z^b)$.\end{theorem}

\begin{proof}$k_z^b$ is a reproducing kernel of $\mathcal{H}(b)$ with respect to the $H^2$ norm. By Corollary \ref{dBRbound}, it reproduces itself with respect to $L^2(\mu)$-boundary.\end{proof}

\begin{remark}It should be noted that Proposition \ref{muboundprop} and Corollary \ref{dBRbound} are previously known. See, for example, Clark's influential paper \cite{Clark72}, Poltoratskii \cite{Pol93}, and Sarason's book \cite{Sar94}. Theorem \ref{dBRkern} is thus simply a formality. However, it can be proven another way, by combining Theorem \ref{DRkernel} with Theorem \ref{kernelthm}, which is to come. \end{remark}

\begin{corollary}\label{dBRsubkern}If $V\subseteq\mathcal{H}(b)$ is a closed subspace and $P_V$ is the orthogonal projection onto $V$, then $P_Vk_z^b\in\mathcal{K}(\mu)$.\end{corollary}

Since the ternary Cantor measure $\mu_3$ is singular, Theorem \ref{dBRkern} shows that $\mathcal{K}(\mu_3)$ is nonempty, despite $\mu_3$ being nonspectral. Corollary \ref{dBRsubkern} shows that $\mathcal{K}(\mu_3)$ contains other members as well. We shall see that there are many more kernels in $\mathcal{K}(\mu_3)$, including some that lie outside $\mathcal{H}(b)$.

\subsection{Wold Decompositions}

Let $b$ be an inner function, and let $\mu$ be its corresponding singular measure. Since the Toeplitz operator $T_b:H^2\to H^2$ is an isometry, and $\mathcal{H}(b)$ is a wandering subspace for $T_b$, the Wold Decomposition Theorem \cite{Wold54} implies
$$H^2=\bigoplus_{n=0}^{\infty} T_b^n\mathcal{H}(b).$$
Although the Wold Decomposition Theorem is well-known \cite{MP88a,LS97,Stessin99}, we offer the following alternative proof for the present situation:

\begin{theorem}Let $\mu$ be a finite singular measure on $[0,1)$, and let $b$ be the inner function corresponding to $\mu$ via the Herglotz representation. Then for any $f\in H^2$, there exists a unique sequence of functions $\{\phi_n\}_{n=0}^{\infty}\subset\mathcal{H}(b)$ such that
$$f=\sum_{n=0}^{\infty}\phi_{n}\cdot b^n.$$
\end{theorem}

\begin{proof}
We know that $k_z^b(w)=\frac{1-\overline{b(z)}b(w)}{1-\overline{z}w}$ is the kernel of $\mathcal{H}(b)$. Thus, $K_z(w)=\overline{b^n(z)}b^n(w)k_z^b(w)\in b^n\mathcal{H}(b)$ for each $n$. (Indeed, it is easy to see it is the kernel of $b^n\mathcal{H}(b)$.) Now, let
$$L=\overline{\text{span}}\{b^n\cdot\phi:n\in\mathbb{N}_0,\phi\in\mathcal{H}(b)\}.$$
For each $k\in\mathbb{N}$, we have that
\[ \sum_{n=0}^{k-1}\overline{b^n(z)}b^n(w)k_z^b(w) = \frac{1-\overline{b^k(z)}b^k(w)}{1-\overline{z}w} \in L. \]
Now, observe that
\begin{align*}\left\lVert\frac{1-\overline{b^k(z)}b^k(w)}{1-\overline{z}w}-\frac{1}{1-\overline{z}w}\right\rVert_{H^2}^2&=\int_{[0,1)}\frac{\lvert b^k(z){b^\ast}^k(e^{2\pi ix})\rvert^2}{\lvert 1-\overline{z}e^{2\pi ix}\rvert^2}\,dx\\
&=\int_{0}^{1}\frac{\lvert b^k(z)\rvert^2}{\lvert 1-\overline{z}e^{2\pi i x}\rvert^2}\,dx\\
&\leq\lvert b(z)\rvert^{2k} C,\end{align*}
where $C=\frac{1}{1-\lvert z\rvert}>0$. Since $b$ is inner, for each fixed $z\in\mathbb{D}$,
$$\lim_{k\rightarrow\infty}\frac{1-\overline{b^k(z)}b^k(w)}{1-\overline{z}w}=\frac{1}{1-\overline{z}w}$$
in the $H^2$-norm. Thus, $\frac{1}{1-\overline{z}w}\in L$ for each fixed $z\in\mathbb{D}$. Since $k_z(w)=\frac{1}{1-\overline{z}w}$ is the kernel of $H^2$, this implies $L=H^2$.

Since $T_{b}$ is an isometry, and $\mathcal{H}(b)$ is the orthogonal complement of the range of $T_{b}$, it follows readily that $b^n\mathcal{H}(b)\perp b^k\mathcal{H}(b)$ for all $n\neq k$ and thus
$$f=\sum_{n=0}^{\infty}\phi_n\cdot b^n,$$
where $\phi_n$ is the unique member of $\mathcal{H}(b)$ such that $\phi_n\cdot b^n$ is the orthogonal projection of $f$ onto $b^n\mathcal{H}(b)$.
\end{proof}

It is easy to show that for $f \in \mathcal{H}(b)$, $(bf)^\ast=f^\ast$, and so every element of $b^{n} \mathcal{H}(b)$ has an $L^2(\mu)$-boundary.  Therefore, if the Wold decomposition of a function $f\in H^2$ is a finite sum, it has an $L^2(\mu)$-boundary. Thus, the Wold Decomposition shows, among other things, that the set of functions in $H^2$ possessing $L^2(\mu)$-boundary is dense.

\begin{proposition}Let $\mu$ be a singular probability measure with corresponding inner function $b$.  Suppose $V_0,V_1,\ldots,V_N$ are  mutually orthogonal closed subspaces of $\mathcal{H}(b)$. Let $k^{(n)}_z(w)$ denote the kernel of $V_n$. Then the space $W=\bigoplus_{n=0}^{N}b^nV_n$ is a reproducing kernel Hilbert space with respect to the norm of $L^2(\mu)$-boundary integration, and its kernel is $K_z:=\sum_{n=0}^{N}\overline{b^n(z)}b^n k_z^{(n)}$. Consequently, $K_z\in\mathcal{K}(\mu)$, and $\mu\in\mathcal{M}(K)$.\end{proposition}

\begin{proof}For any $f\in W$, we may write $f=f_0+bf_1+b^2f_2+\ldots+b^Nf_N$, where $f_n\in V_n$. Then observe that by mutual orthogonality of the spaces $\mathcal{H}(b),b\mathcal{H}(b),b^2\mathcal{H}(b),\ldots,b^N\mathcal{H}(b)$ in $H^2$, we have
\begin{align*}\lVert f\rVert_{H^2}^2&=\sum_{n=0}^{N}\lVert b^nf_n\rVert_{H^2}^2=\sum_{n=0}^{N}\lVert f_n\rVert_{H^2}^2=\sum_{n=0}^{N}\lVert f_n\rVert_{\mathcal{H}(b)}^2=\sum_{n=0}^{N}\lVert f_n^\ast\rVert_{\mu}^2.\end{align*}
By mutual orthogonality of the spaces $V_0,V_1,\ldots,V_N$ in $\mathcal{H}(b)$, the $f_n$ are orthogonal in $\mathcal{H}(b)$, and hence by Corollary \ref{dBRbound} the $f_n^\ast$ are orthogonal in $L^2(\mu)$. Hence,
\begin{align*}\sum_{n=0}^{N}\lVert f_n^\ast\rVert_{\mu}^2=\sum_{n=0}^{N}\lVert(b^n f)^\ast\rVert_{\mu}^2=\left\lVert \sum_{n=0}^{N}(b^nf_n)^\ast\right\rVert_{\mu}^2=\left\lVert\left(\sum_{n=0}^{N}b^nf_n\right)^\ast\right\rVert_{\mu}^2=\lVert f^\ast\rVert_{\mu}^2.\end{align*}
This shows that the $H^2$ norm and the $L^2(\mu)$-boundary norm are equal on $W$. Hence, the inner products are equal as well by the polarization identity. The proof is completed by noting that by orthogonality,
\begin{align*}\left\langle f,\sum_{n=0}^{N}\overline{b^n(z)}b^nk_z^{(n)}\right\rangle_{H^2}&=\left\langle\sum_{m=0}^{N}b^mf_m,\sum_{n=0}^{N}\overline{b^n(z)}b^nk_z^{(n)}\right\rangle_{H^2}\\
&=\sum_{n=0}^{N}b^n(z)\langle f_n,k_z^{(n)}\rangle_{H^2}\\
&=f(z)\end{align*}
\end{proof}

\section{Kernels in $\mathcal{K}(\mu)$ That Are not $H^2$ Kernels}

We have seen that for a singular probability measure $\mu$, there are many kernels in $\mathcal{K}(\mu)$, obtained by projecting the Szeg\H{o} kernel onto appropriate subspaces of $H^2$.  We now turn to showing that there are many kernels in $\mathcal{K}(\mu)$ which are not obtained in this way, and in fact the kernels will generate subspaces of $H^2$ for which the norm defined by the kernel is not identical to the norm in $H^2$.  The following definition will be convenient in our subsequent discussions:

\begin{definition}Given a Hilbert space $\mathbb{H}$ and two sequences $\{x_n\}_{n=0}^{\infty}$ and $\{y_n\}_{n=0}^{\infty}$ in $\mathbb{H}$, if we have
\begin{equation}\sum_{n=0}^{\infty}\langle f,x_n\rangle y_n=f\end{equation}
with convergence in norm for all $f\in\mathbb{H}$, then $\{x_n\}_{n=0}^{\infty}$ is said to be dextrodual to $\{y_n\}_{n=0}^{\infty}$ (or, ``a dextrodual of $\{y_n\}_{n=0}^{\infty}$''), and $\{y_n\}_{n=0}^{\infty}$ is said to be levodual to $\{x_n\}_{n=0}^{\infty}$.\end{definition}

In the parlance of frame theory, if $S_y$ is the synthesis operator of $\{y_n\}$ and $A_x$ is the analysis operator of $\{x_n\}$, then $\{x_n\}$ is dextrodual to $\{y_n\}$ if $S_yA_x=I$. However, a sequence does not need to be a frame to have a dextrodual. For example, $\{e_n\}_{n=0}^{\infty}$ is not even Bessel in $L^2(\mu)$ for $\mu$ a singular measure, but \eqref{gndextro} shows that the Parseval frame $\{g_n\}_{n=0}^{\infty}$ is dextrodual to $\{e_n\}_{n=0}^{\infty}$.

For a singular probability measure $\mu$ on $[0,1)$, in \cite{HW15} we demonstrate a large class of dextroduals of $\{e_n\}_{n=0}^{\infty}$ via the following theorem:

\begin{thmalpha}[HW15] \label{T:reproduce}
Let $\mu$ be a singular probability measure on $[0,1)$. Let $\nu$ be another singular probability measure on $[0,1)$ such that $\nu\perp\mu$. Let $0<\eta\leq1$, and define $\lambda:=\eta\mu+(1-\eta)\nu$. Let $\{h_n\}$ be the sequence associated to $\{e_n\}$ in $L^2(\lambda)$ via the Kaczmarz algorithm in Equation \eqref{gs}. Then for all $f\in L^2(\mu)$,
\begin{equation} \label{Eq:reproduce}
f=\sum_{n=0}^{\infty}{\langle f,\eta h_n\rangle}_{\mu} e_n
\end{equation}
in the $L^2(\mu)$ norm. Moreover, if $\lambda^\prime=\eta^\prime\mu+(1-\eta^\prime)\nu^\prime$ also satisfies the hypotheses, then $\lambda^\prime\neq\lambda$ implies $\{\eta^\prime h_n^\prime\}\neq\{\eta h_n\}$ in $L^2(\mu)$. \end{thmalpha}

It is worth noting that since the sequence $\{h_n\}$ is a Parseval frame in $L^2(\lambda)$, for any $f\in L^2(\mu)$, with $\tilde{f}=f\cdot\chi_{\text{supp}(\mu)}\in L^2(\lambda)$, we have
\[ \sum_{n=0}^{\infty}\lvert\langle f,\eta h_n\rangle_\mu\rvert^2 = \sum_{n=0}^{\infty}\left\lvert\left\langle\tilde{f},h_n\right\rangle_{\lambda}\right\rvert^2 
=\left\lVert\tilde{f}\right\rVert^2_{\lambda} 
=\eta\left\lVert f\right\rVert^2_{\mu}.
\]

Hence, $\{\eta h_n\}$ is a tight frame in $L^2(\mu)$ with frame bound $\eta$. If $\eta=1$, then $\{\eta h_n\}=\{g_n\}$ is the canonical sequence associated to $\mu$ by $\eqref{gs}$, a Parseval frame that by $\eqref{gndextro}$ is a dextrodual of $\{e_n\}$ in $L^2(\mu)$. When $\eta<1$, this theorem gives us a number of other dextroduals of $\{e_n\}$ that, while not Parseval frames, are tight frames. These dextroduals then provide us with reproducing kernels analogous to Theorem \ref{DRkernel}.

\begin{theorem}\label{kernelthm}Let $\mu$ be a Borel measure on $[0,1)$. Let $\{h_n\}\subset L^2(\mu)$ be a Bessel sequence that is dextrodual to $\{e_n\}$. Then for each fixed $z\in\mathbb{D}$,
$$K_z(w):=\sum_{m}\sum_{n}\langle h_n,h_m\rangle_\mu\overline{z}^nw^m$$
is a well-defined function on $\mathbb{D}$. $K_z(w)\in H^2$ and possesses an $L^2(\mu)$-boundary function $K^\ast_z$. Moreover,
$$K_z(w)=\langle K^\ast_z,K^\ast_w\rangle_{\mu},$$
and thus $K \in \mathcal{K}(\mu)$.
\end{theorem}

\begin{proof}Fix $z\in\mathbb{D}$. Let $N\in\mathbb{N}_0$, and suppose $n>m\geq N$. Then since $\{h_n\}$ is Bessel, we have
$$\left\lVert\sum_{k=0}^{n}\overline{z}^kh_k-\sum_{k=0}^{m}\overline{z}^kh_k\right\rVert_{\mu}=\left\lVert\sum_{k=m+1}^{n}\overline{z}^kh_k\right\rVert_{\mu}\leq B\sqrt{\sum_{k=m+1}^{n}\lvert z\rvert^{2k}}\leq B\sqrt{\sum_{k=N}^{\infty}\lvert z\rvert^{2k}}.$$
As $N\rightarrow\infty$, the right side goes to $0$, which shows that the sequence $\displaystyle\left\{\sum_{k=0}^{n}\overline{z}^kh_k\right\}_n$ is Cauchy and hence convergent in $L^2(\mu)$. By continuity of the inner product in $L^2(\mu)$, we then have
\begin{align*}K_z(w)&:=\sum_{m}\sum_{n}\langle h_n,h_m\rangle\overline{z}^nw^m\\
&=\sum_{m}\left\langle\sum_{n}\overline{z}^n h_n,h_m\right\rangle w^m.\end{align*}
Observe that since $\{h_n\}$ is Bessel,
\begin{align*}\sum_{m=0}^{\infty}\left\lvert\left\langle\sum_{n}\overline{z}^n h_n,h_m\right\rangle\right\rvert^2\leq B^\prime\left\lVert\sum_{n}\overline{z}^nh_n\right\rVert^2_{\mu}<\infty,\end{align*}
which shows that $K_z(w)\in H^2$. Define $K_z^\ast\in L^2(\mu)$ by $K_z^\ast=\sum_{n}\overline{z}^nh_n$. Because $\{h_n\}$ is dextrodual to $\{e_n\}$, we have
$$K_z^\ast:=\sum_{n}\overline{z}^nh_n=\sum_{m}\left\langle\sum_{n}\overline{z}^nh_n,h_m\right\rangle e_m.$$
A summable series in a normed linear space is Abel summable. Hence, for all $0<r\leq1$, we have that
$$\sum_{m}r^m\left\langle\sum_{n}\overline{z}^nh_n,h_m\right\rangle e_m$$
converges in $L^2(\mu)$, and
\begin{align*}&\lim_{r\rightarrow1^{-}}\left\lVert\sum_{m}\left\langle\sum_{n}\overline{z}^nh_n,h_m\right\rangle e_m-\sum_{m}r^m\left\langle\sum_{n}\overline{z}^nh_n,h_m\right\rangle e_m\right\rVert_{\mu}\\
&=\lim_{r\rightarrow1^{-}}\left\lVert\sum_{n}\overline{z}^nh_n-\sum_{m}r^m\left\langle\sum_{n}\overline{z}^nh_n,h_m\right\rangle e_m\right\rVert_{\mu}\\
&=0.\end{align*}
Since
$$K_z(re^{2\pi ix})=\sum_{m}\left\langle\sum_{n}\overline{z}^nh_n,h_m\right\rangle r^me^{2\pi imx},$$
the above shows that for each $0<r<1$, $K_z(re^{2\pi ix})\in L^2(\mu)$ with respect to the variable $x$, and $K_z^\ast$ is an $L^2(\mu)$-boundary function of $K_z(w)$. We compute that
\[ \langle K_z^\ast,K_w^\ast\rangle=\left\langle\sum_{n}\overline{z}^nh_n,\sum_{m}\overline{w}^mh_m\right\rangle
=\sum_{m}\sum_{n}\left\langle h_n,h_m\right\rangle\overline{z}^nw^m
=K_z(w).
\]
\end{proof}

The following result generalizes the construction in Theorem B and gives us dextroduals of $\{e_n\}$ that are also frames:

\begin{theorem}\label{DDfromAC}Suppose $\mu$ and $\lambda$ are singular probability measures on $[0,1)$ such that $\mu<<\lambda$, and suppose there exist constants $A$ and $B$ such that
$$0<A\leq\frac{d\mu}{d\lambda}\leq B$$
on $\text{supp}\left(\frac{d\mu}{d\lambda}\right):=\left\{x\in[0,1)~|~\frac{d\mu}{d\lambda}(x)\neq0\right\}$.
If $\{h_n\}$ is the canonical sequence associated to $\lambda$ by $\eqref{gs}$, then $\left\{\frac{h_n}{\frac{d\mu}{d\lambda}}\right\}$ is dextrodual to $\{e_n\}$ in $L^2(\mu)$. Moreover, $\left\{\frac{h_n}{\frac{d\mu}{d\lambda}}\right\}$ is a frame in $L^2(\mu)$ with bounds no worse than $\frac{1}{B}$ and $\frac{1}{A}$.\end{theorem}

\begin{proof}Let $M:=\text{supp}\left(\frac{d\mu}{d\lambda}\right)$, and define $\tilde{f}:=f\cdot\chi_{M}$. First, we observe that
\begin{align*}\int_{[0,1)}\left\lvert \frac{h_n}{\frac{d\mu}{d\lambda}}\right\rvert^2\,d\mu
=\int_{M}\frac{\lvert h_n\rvert^2}{\left(\frac{d\mu}{d\lambda}\right)^2}\frac{d\mu}{d\lambda}\,d\lambda
\leq\int_{M}\frac{\lvert h_n\rvert^2}{A}\,d\lambda
<\infty.
\end{align*}
This shows that $\left\{\frac{h_n}{\frac{d\mu}{d\lambda}}\right\}\subset L^2(\mu)$. Now, suppose $f\in L^2(\mu)$. Then
\begin{align*}&A\int_{[0,1)}\lvert\tilde{f}\rvert^2\,d\lambda
\leq\int_{[0,1)}\left\lvert\tilde{f}\right\rvert^2\frac{d\mu}{d\lambda}\,d\lambda
\leq B\int_{[0,1)}\lvert\tilde{f}\rvert^2\,d\lambda.
\end{align*}
Therefore,
$$\frac{1}{B}\int_{[0,1)}\lvert f\rvert^2\,d\mu\leq\int_{[0,1)}\left\lvert\tilde{f}\right\rvert^2\,d\lambda\leq\frac{1}{A}\int_{[0,1)}\lvert f\rvert^2\,d\mu<\infty.$$
Thus, $f\in L^2(\mu)\implies\tilde{f}\in L^2(\lambda)$. Now, we compute that for any $f\in L^2(\mu)$,
\begin{align*}\left\langle f,\frac{h_n}{\frac{d\mu}{d\lambda}}\right\rangle_\mu
=\int_{[0,1)}f\cdot\overline{\left(\frac{h_n}{\frac{d\mu}{d\lambda}}\right)}\frac{d\mu}{d\lambda}\,d\lambda
=\int_{[0,1)}\tilde{f}\cdot\overline{h_n}\,d\lambda
=\langle\tilde{f},h_n\rangle_\lambda.
\end{align*}
Then because $\{h_n\}$ is a Parseval frame in $L^2(\lambda)$, the previous two computations show that
\begin{align}\label{Jequivnorm}
\frac{1}{B}\lVert f\rVert_{\mu}^2\leq\lVert\tilde{f}\rVert_{\lambda}^2=\sum_{n=0}^{\infty}\lvert\langle \tilde{f},h_n\rangle_\lambda\rvert^2=\sum_{n=0}^{\infty}\left\lvert\left\langle f,\frac{h_n}{\frac{d\mu}{d\lambda}}\right\rangle_\mu\right\rvert^2\leq\frac{1}{A}\lVert f\rVert_\mu^2,
\end{align}
and hence that $\left\{\frac{h_n}{\frac{d\mu}{d\lambda}}\right\}$ is a frame in $L^2(\mu)$ with bounds no worse than $\frac{1}{B}$ and $\frac{1}{A}$.
Now, we have that
\begin{align*}
\lim_{k\rightarrow\infty}\left\lVert f-\sum_{n=0}^{k}\left\langle f,\frac{h_n}{\frac{d\mu}{d\lambda}}\right\rangle_\mu e_n\right\rVert_{\mu}^2
&=\lim_{k\rightarrow\infty}\int_{[0,1)}\left\lvert \tilde{f}-\sum_{n=0}^{k}\left\langle \tilde{f},h_n\right\rangle_\lambda e_n\right\rvert^2\frac{d\mu}{d\lambda}\,d\lambda\\
&\leq B \lim_{k\rightarrow\infty}\left\lVert\tilde{f}-\sum_{n=0}^{k}\left\langle \tilde{f},h_n\right\rangle_\lambda e_n\right\rVert^2_{\lambda}\\
&=0,
\end{align*}
which shows that $\left\{\frac{h_n}{\frac{d\mu}{d\lambda}}\right\}$ is dextrodual to $\{e_n\}$ in $L^2(\mu)$.
\end{proof}

Combining Theorems \ref{kernelthm} and \ref{DDfromAC}, we obtain kernels which are in $\mathcal{K}(\mu)$.  These kernels, however, are not the restriction of the Szeg\H{o} kernel on some subspace of $H^2$, as we shall now demonstrate.

\begin{proposition}\label{Jtheorem}Assume the hypotheses of Theorem \ref{DDfromAC}. The set
\begin{equation} \label{eq:jlambda}
J(\lambda)=\left\{\sum_{n=0}^{\infty}\left\langle f,\frac{h_n}{\frac{d\mu}{d\lambda}}\right\rangle_{\mu}z^n~\middle|~f\in L^2(\mu)\right\}
\end{equation}
is a closed linear subspace of $H^2$. If $\lambda\left([0,1)\setminus\text{supp}\left(\frac{d\mu}{d\lambda}\right)\right)>0$, then $J$ is not backward-shift-invariant. If $\lambda\left([0,1)\setminus\text{supp}\left(\frac{d\mu}{d\lambda}\right)\right)=0$, then $J=\mathcal{H}(\lambda)$.
\end{proposition}

\begin{proof}By Theorem \ref{DDfromAC}, $\left\{\frac{h_n}{\frac{d\mu}{d\lambda}}\right\}$ is a frame in $L^2(\mu)$, so that $J(\lambda)$ is a subset of $H^2$. It is clearly a linear subspace of $H^2$ by sesquilinearity of the inner product in $L^2(\mu)$. By Equation \eqref{Jequivnorm}, 
$$\lVert f\rVert_\mu^2\simeq\left\lVert\sum_{n=0}^{\infty}\left\langle f,\frac{h_n}{\frac{d\mu}{d\lambda}}\right\rangle_\mu z^n\right\rVert_{H^2}^2.$$ Thus, $J(\lambda)$ is a closed subset of $H^2$ by virtue of $L^2(\mu)$ being closed.

Now, let $S^\ast$ denote the backward shift operator acting on $H^2$. Let $\{\alpha_n\}$ be the sequence from \eqref{gnformula} satisfying $h_n=\sum_{i=0}^{n}\overline{\alpha_{n-i}}e_i$. Observe that for all $n\in\mathbb{N}_0$,
$$e_1h_n=e_1\sum_{i=0}^{n}\overline{\alpha_{n-i}}e_i=\sum_{i=0}^{n}\overline{\alpha_{n-i}}e_{i+1}=\sum_{i=1}^{n+1}\overline{\alpha_{n+1-i}}e_i=h_{n+1}-\overline{\alpha_{n+1}}e_0.$$
For any $f\in L^2(\lambda)$, it is trivial to see that
$$\frac{f-\langle f,e_0\rangle_{\lambda}e_0}{e_1}\in L^2(\lambda).$$
We compute that
\begin{align*}\left\langle \frac{f-\langle f,e_0\rangle_{\lambda}e_0}{e_1},h_n\right\rangle_{\lambda}
&=\langle f,e_1h_n\rangle_{\lambda}-\langle f,e_0\rangle_{\lambda}\langle e_0,e_1 h_n\rangle_{\lambda}\\
&=\langle f,h_{n+1}\rangle_{\lambda}-\alpha_{n+1}\langle f,e_0\rangle_{\lambda}-\langle f,e_0\rangle_{\lambda}\langle e_0,h_{n+1}\rangle_{\lambda}+\langle f,e_0\rangle_{\lambda}\alpha_{n+1}\langle e_0,e_0\rangle_{\lambda}\\
&=\langle f,h_{n+1}\rangle_{\lambda},\end{align*}
because $\langle e_0,e_0\rangle_{\lambda}=1$ and $\langle e_0,h_{n+1}\rangle_{\lambda}=0$ for all $n\geq0$. Thus,
$$S^\ast\left(\sum_{n=0}^{\infty}\langle f,h_n\rangle_{\lambda}z^n\right)=\sum_{n=0}^{\infty}\left\langle \frac{f-\langle f,e_0\rangle_{\lambda}e_0}{e_1},h_n\right\rangle_{\lambda}z^n.$$
As before, let $M=\text{supp}\left(\frac{d\mu}{d\lambda}\right)$. For any $f\in L^2(\mu)$, $\tilde{f}=f\cdot\chi_{M}$ is the unique member of $L^2(\lambda)$ such that $\left\langle f,\frac{h_n}{\frac{d\mu}{d\lambda}}\right\rangle_{\mu}=\langle\tilde{f}, h_n\rangle_{\lambda}$ for all $n\geq0$. We therefore have
\begin{equation}\label{backwardshiftexpo}S^\ast\left(\sum_{n=0}^{\infty}\left\langle f,\frac{h_n}{\frac{d\mu}{d\lambda}}\right\rangle_{\mu}z^n\right)=\sum_{n=0}^{\infty}\left\langle\frac{\tilde{f}-\left\langle\tilde{f},e_0\right\rangle_{\lambda}e_0}{e_1},h_n\right\rangle_{\lambda} z^n.\end{equation}
Observe that on $[0,1)\setminus M$,
$$\frac{\tilde{f}-\left\langle\tilde{f},e_0\right\rangle_{\lambda}e_0}{e_1}=-\left\langle\tilde{f},e_0\right\rangle_{\lambda}e_{-1}.$$
Let us examine the particular case in which $f=e_0\in L^2(\mu)$. We have
$$\left\langle\tilde{f},e_0\right\rangle_{\lambda}=\lambda(M),$$
so that on $[0,1)\setminus M$,
$$\frac{\tilde{f}-\left\langle\tilde{f},e_0\right\rangle_{\lambda}e_0}{e_1}=-\lambda(M) e_{-1}.$$
Since $\lambda(M)>0$, $-\lambda(M) e_{-1}=0$ $\lambda$-a.e.~on $[0,1)\setminus M$ if and only if $\lambda([0,1)\setminus M)=0$. It follows that if $\lambda([0,1)\setminus M)>0$, there does not exist $w\in L^2(\mu)$ such that
$$\frac{\tilde{f}-\left\langle\tilde{f},e_0\right\rangle_{\lambda}e_0}{e_1}=w\cdot\chi_{M}=\tilde{w}$$
in $L^2(\lambda)$. Hence, if $\lambda([0,1)\setminus M)>0$, then $J(\lambda)$ is not backward-shift-invariant.

If $\lambda([0,1)\setminus M)=0$, then it is easy to see that for all $f\in L^2(\mu)$, $\tilde{f}=f$ in $L^2(\lambda)$, and
$$A\lVert f\rVert_{\lambda}^2\leq\lVert f\rVert_{\mu}^2\leq B\lVert f\rVert_{\lambda}^2,$$
so that $L^2(\lambda)=L^2(\mu)$ as sets of functions. Thus
$$J(\lambda)=\left\{\sum_{n=0}^{\infty}\left\langle f,\frac{h_n}{\frac{d\mu}{d\lambda}}\right\rangle_{\mu}z^n~\middle|~f\in L^2(\mu)\right\}=\left\{\sum_{n=0}^{\infty}\left\langle f,h_n\right\rangle_{\lambda}z^n~\middle|~f\in L^2(\lambda)\right\}=\mathcal{H}(\lambda).$$
\end{proof}

The following lemma is a consequence of Abel summability.

\begin{lemma} \label{L:Jboundary}
Every $F \in J(\lambda)$ possesses an $L^2(\mu)$-boundary; indeed, if 
\[ F(z) = \sum_{n=0}^{\infty}\left\langle f,\frac{h_n}{\frac{d\mu}{d\lambda}}\right\rangle_{\mu}z^n, \]
then $f$ is the $L^2(\mu)$-boundary of $F$.
\end{lemma}


\begin{definition}
For $F,G \in J(\lambda)$, define the inner product as:
\begin{equation} \label{Eq:Jinner}
\langle F,G\rangle_{J}:=\int_{0}^{1}F^{*}(x)\overline{G^{*}(x)}\,d\mu(x).
\end{equation}
\end{definition}

\begin{corollary}\label{Jspacethm}The space $J(\lambda)$ in Equation (\ref{eq:jlambda}), when equipped with the inner product in Equation (\ref{Eq:Jinner}), is a reproducing kernel Hilbert space with kernel
\begin{equation*}K_z(w):=\sum_{m}\sum_{n}\left\langle\frac{h_n}{\frac{d\mu}{d\lambda}}, \frac{h_m}{\frac{d\mu}{d\lambda}}\right\rangle_\mu\overline{z}^nw^m,\end{equation*}
and the norm of this space is equivalent to the usual Hardy space norm.
\end{corollary}

\begin{proof}  By Lemma \ref{L:Jboundary}, the mapping $F \mapsto F^{*}$ is well-defined from $J(\lambda)$ to $L^2(\mu)$, and by definition is a bijection.  When $J(\lambda)$ is equipped with the norm induced by $\langle\cdot,\cdot\rangle_{J}$, the mapping is an isometry.  It follows that  $J(\lambda)$ is complete. 

Let $F\in J(\lambda)$. Observe that by Theorem \ref{DDfromAC},
\begin{equation*}\frac{1}{B}\lVert F\rVert_{J}^2=\frac{1}{B}\lVert F^\ast\rVert_{\mu}^2 
\leq\sum_{n=0}^{\infty}\left\lvert\left\langle F^\ast,\frac{h_n}{\frac{d\mu}{d\lambda}}\right\rangle_{\mu}\right\rvert^2=\lVert F \rVert_{H^2}^2
\leq\frac{1}{A}\lVert F^\ast\rVert_{\mu}^2=\frac{1}{A}\lVert F\rVert_{J}^2,\end{equation*}
showing that $\lVert \cdot\rVert_{J}$ and $\lVert\cdot\rVert_{H^2}$ are equivalent norms on $J(\lambda)$.

Because $\{\frac{h_n}{\frac{d\mu}{d\lambda}}\}$ is a frame on $L^2(\mu)$, it is Bessel on $L^2(\mu)$, and it follows from Theorem \ref{kernelthm} that $K_z(w)$ is well-defined on $\mathbb{D}$, possesses an $L^2(\mu)$-boundary, and reproduces itself with respect to that boundary. For each $z\in\mathbb{D}$, as shown in the proof of Theorem \ref{kernelthm}, $\sum_{n=0}^{\infty}\overline{z}^n\frac{h_n}{\frac{d\mu}{d\lambda}}\in L^2(\mu)$, and
\begin{equation*}K_z(w)=\sum_{m=0}^{\infty}\left\langle\sum_{n=0}^{\infty}\overline{z}^n\frac{h_n}{\frac{d\mu}{d\lambda}},\frac{h_m}{\frac{d\mu}{d\lambda}}\right\rangle_\mu w^m.\end{equation*}
Thus, $K_z(w)\in J(\lambda)$ for each $z\in\mathbb{D}$. 

It remains only to show that $\{K_z(w):z\in\mathbb{D}\}$ is complete in $J(\lambda)$. Let $b$ denote the inner function corresponding to the measure $\lambda$ from Proposition \ref{Jtheorem} (where $\mu=\frac{d\mu}{d\lambda}\,d\lambda$). Since $\{h_n\}$ is the canonical sequence associated to $\lambda$ via the Kaczmarz algorithm in \eqref{gs}, it follows as in the proof of Theorem \ref{DRkernel} that
\begin{equation*}\sum_{n=0}^{\infty}\overline{z}^n\frac{h_n}{\frac{d\mu}{d\lambda}}=\frac{1}{\frac{d\mu}{d\lambda}}(1-\overline{b(z)})\frac{1}{1-\overline{z}e_1},\end{equation*}
where the convergence occurs absolutely in norm. Therefore, if $\left\{\frac{1}{\frac{d\mu}{d\lambda}}\cdot\frac{1}{1-\overline{z}e_1}:z\in\mathbb{D}\right\}$ is linearly dense in $L^2(\mu)$, then by linearity $\{K_z(w):z\in\mathbb{D}\}$ is linearly dense in $J(\lambda)$. So suppose $f\in L^2(\mu)$. Then $V_\mu f(z)\in\mathcal{H}(d)$, where $d$ is the inner function corresponding to $\mu$. The kernel functions of $\mathcal{H}(d)$ are of the form
\begin{equation*}k^{d}_z(w)=\frac{1-\overline{d(z)}d(w)}{1-\overline{z}w}.\end{equation*}
These kernels are, of course, linearly dense in $\mathcal{H}(d)$, and since the norm on $\mathcal{H}(d)$ corresponds to the $L^2(\mu)$-boundary norm, we must have that the boundary functions of the kernels, $\{\left({k_z^d}\right)^\ast:z\in\mathbb{D}\}$, are linearly dense in $L^2(\mu)$. As remarked in \cite{Pol93}, the radial limits of $d$ satisfy $d^\ast(e^{2\pi ix})=1$ for $\mu$-almost all $x$. Thus the $L^2(\mu)$-boundary of $k_z^d(w)$ is
\begin{equation*}(k_z^d)^\ast(x)=\frac{1-\overline{d(z)}}{1-\overline{z}e^{2\pi ix}}.\end{equation*}
Suppressing the constant $1-\overline{d(z)}$, we see that $\left\{\frac{1}{1-\overline{z}e_1}:z\in\mathbb{D}\right\}$ is linearly dense in $L^2(\mu)$. Because $\frac{d\mu}{d\lambda}\cdot f\in L^2(\mu)$ for any $f\in L^2(\mu)$, this implies $\left\{\frac{1}{\frac{d\mu}{d\lambda}}\cdot\frac{1}{1-\overline{z}e_1}:z\in\mathbb{D}\right\}$ is linearly dense in $L^2(\mu)$, which completes the proof.
\end{proof}
We now obtain our desired result concerning the multitude of elements of $\mathcal{K}(\mu)$.
\begin{corollary}Under the hypotheses and notations of Theorem \ref{DDfromAC} and Corollary \ref{Jspacethm}, we have that $K\in\mathcal{K}(\mu)$, and $\mu\in\mathcal{M}(K)$.\end{corollary}

We close this section with a few observations.  By Proposition \ref{Jtheorem}, if $\lambda\left([0,1)\setminus\text{supp}\left(\frac{d\mu}{d\lambda}\right)\right)>0$, then because $J(\lambda)$ is not backward-shift-invariant, the kernel $K$ in Corollary \ref{Jspacethm}
is a positive matrix such that $\mu\in\mathcal{M}(K)$, but $J(\lambda)\neq\mathcal{H}(b)$. Indeed $J(\lambda)\neq\mathcal{H}(u)$ for any inner function $u$, because $\mathcal{H}(u)$ is backward-shift-invariant.  Moreover, the norm on $J(\lambda)$ that yields $K$ as the kernel is equivalent, but in general not equal, to the Hardy space norm.  Only if $A=B=1$ in Theorem \ref{DDfromAC} will the norms be equal.  As such, the kernel $K$ is not simply the projection of the Szeg\H{o} kernel onto $J(\lambda)$.  In the case that $\lambda\left([0,1)\setminus\text{supp}\left(\frac{d\mu}{d\lambda}\right)\right)=0$, the equality of $J(\lambda) = \mathcal{H}(\lambda)$ at the end of Proposition \ref{Jtheorem} is as sets; as just observed, the norms on these spaces need not be equal.

\section{The Set $\mathcal{M}(K)$}

Starting with a singular probability measure $\mu$, we have seen large classes of positive matrices $K_z(w)$ that reproduce with respect to $L^2(\mu)$-boundary integration. Reproducing in this way potentially has desirable application, but it may happen in practice that we are more tied to a particular positive matrix than we are a measure. Thus, it is natural for us to ask a question in the opposite direction: Given a positive matrix $K\subset H^2(\mathbb{D})$, for which Borel measures $\mu$ does $K_z(w)$ reproduce with respect to $L^2(\mu)$-boundary integration? In other words, which measures are in $\mathcal{M}(K)$? For a given $K$, it is \textit{a priori} possible that $\mathcal{M}(K)=\varnothing$, though we know of no examples yet. As we have seen, though, this is thankfully not always the case, and the following results give us some more insight.

\begin{theorem}Let $V$ be a closed subspace of $H^2$, and let $K$ be the reproducing kernel of $V$. If
$$\overline{\cup_{n=0}^{\infty}{S^\ast}^nV}\neq H^2,$$
then there exists a singular measure $\mu\in\mathcal{M}(K)$.  Indeed, to each inner function $b$ with $b(0) = 0$ there corresponds a distinct such measure.  \end{theorem}
\begin{proof}$\overline{\cup_{n=0}^{\infty}{S^\ast}^nV}$ is the smallest closed $S^\ast$-invariant subspace containing $V$. Every proper closed $S^\ast$-invariant subspace of $H^2$ is a de Branges-Rovnyak space $\mathcal{H}(u)$ for some inner function $u$. Let $b$ be an inner function such that $b(0) = 0$, and let $\mu$ be the singular probability measure corresponding to $ub$. Then by Corollary \ref{dBRbound}, the $H^2$ norm on $\mathcal{H}(ub)$ is equal to the norm of $L^2(\mu)$-boundary integration. Thus since $V \subset \mathcal{H}(u) \subset \mathcal{H}(ub)$ and $K$ reproduces with respect to the $H^2$ norm in $\mathcal{H}(ub)$, it reproduces with respect to the $L^2(\mu)$-boundary norm. Hence, $\mu\in\mathcal{M}(K)$.\end{proof}

\begin{lemma}\label{affinehulllemma}Let $\nu$ and $\mu$ be finite Borel measures on $[0,1)$, and suppose $\nu=\nu_a+\nu_s$ is the Lebesgue decomposition of $\nu$ with respect to $\mu$. If $\mu,\nu\in\mathcal{M}(K)$ and $\frac{d\nu_a}{d\mu}$ is bounded, then the affine hull of $\nu$ and $\mu$ intersected with the set of nonnegative Borel measures is contained in $\mathcal{M}(K)$.
\end{lemma}

\begin{proof}Suppose $\nu,\mu\in\mathcal{M}(K)$ with $\frac{d\nu_a}{d\mu}\leq\beta$, and let $\lambda\in\mathbb{R}$ such that $\lambda\mu+(1-\lambda)\nu$ is a nonnegative Borel measure. For each $z$, let $K_{z,\mu}^\ast:[0,1)\to\mathbb{C}$ be a representative of the equivalence class of the $L^2(\mu)$-boundary of $K_z$, and likewise let $K_{z,\nu}^\ast$ be a representative of the equivalence class of the $L^2(\nu)$-boundary of $K_z$. Since $\nu_s\perp\mu$, there exist disjoint Borel subsets $A$ and $B$ of $[0,1)$ such that $A\cup B=[0,1)$, $\nu_s(E)=0$ for all $E\subseteq B$, and $\mu(E)=0$ for all $E\subseteq A$.

Define $H_z=K_{z,\mu}^{\ast}\cdot\chi_{B}+K_{z,\nu}^\ast\cdot\chi_{A}$. It is obvious that $H_z\equiv_{L^2(\mu)}K_{z,\mu}^\ast$. We claim that $H_z\equiv_{L^2(\nu)}K_{z,\nu}^\ast$ as well. We compute:

\begin{align*}\lim_{r\to1^-} &\int_{0}^{1}\lvert K_z(re_x)-H_z(x)\rvert^2\,d\nu(x)\\
&=\lim_{r\to1^-}\left(\int_{[0,1)}\lvert K_z(re_x)-H_z(x)\rvert^2\frac{d\nu_a}{d\mu}\,d\mu(x)+\int_{[0,1)}\lvert K_z(re_x)-H_z(x)\rvert^2\,d\nu_s(x)\right)\\
&=\lim_{r\to1^-}\left(\int_{B}\lvert K_z(re_x)-H_z(x)\rvert^2\frac{d\mu_a}{d\mu}\,d\mu(x)+\int_{A}\lvert K_z(re_x)-H_z(x)\rvert^2\,d\nu_s(x)\right)\\
&=\lim_{r\to1^-}\left(\int_{B}\lvert K_{z}(re_x)-K_{z,\mu}^\ast(x)\rvert^2\frac{d\nu_a}{d\mu}\,d\mu(x)+\int_{A}\lvert K_z(re_x)-K_{z,\nu}^\ast(x)\rvert^2\,d\nu_s(x)\right).\end{align*}
Since $\mu(A)=0$, we have $\nu_a(A)=0$, and hence
$$\lim_{r\to1^-}\int_A\lvert K_z(re_x)-K_{z,\nu}^\ast(x)\rvert^2\,d\nu_s(x)=\lim_{r\to1^-}\int_{A}\lvert K_z(re_x)-K_{z,\nu}^\ast(x)\rvert^2\,d\nu(x)=0.$$
We also have
\begin{align*}\lim_{r\to1^-}\int_{0}^{1}\lvert K_z(re_x)-H_z(x)\rvert^2\,d\nu(x)&=\lim_{r\to1^-}\int_{B}\lvert K_z(re_x)-K_{z,\mu}^\ast(x)\rvert^2\frac{d\nu_a}{d\mu}\,d\mu(x)\\
&\leq\beta\lim_{r\to1^-}\int_{B}\lvert K_z(re_x)-K_{z,\mu}^\ast(x)\rvert^2\,d\mu(x)\\
&=0
\end{align*}
Therefore, $H_z\equiv_{L^2(\nu)}K_{z,\nu}^\ast$. Now observe that
\begin{align*}\lim_{r\to1^-} &\int_{[0,1)}\lvert K_z(re_x)-H_z(x)\rvert^2\,d[\lambda\mu+(1-\lambda)\nu](x)\\
&=\lim_{r\to1^-}\left(\lambda\int_{[0,1)}\lvert K_z(re_x)-H_z(x)\rvert^2\,d\mu(x)+(1-\lambda)\int_{[0,1)}\lvert K_z(re_x)-H_z(x)\rvert^2\,d\nu(x)\right)\\
&=\lim_{r\to1^-}\left(\lambda\int_{[0,1)}\lvert K_z(re_x)-K_{z,\mu}^\ast\rvert^2(x)\,d\mu(x)+(1-\lambda)\int_{[0,1)}\lvert K_z(re_x)-K_{z,\nu}^\ast(x)\rvert^2\,d\nu(x)\right)\\
&=0\end{align*}
It follows that $H_z$ is the $L^2(\rho)$-boundary of $K_z$, where $\rho=\lambda\mu+(1-\lambda)\nu$. We see that
\begin{align*}\int_{[0,1)}H_z\overline{H_w}\,d\rho&=\lambda\int_{[0,1)}H_z\overline{H_w}\,d\mu+(1-\lambda)\int_{0}^{1}H_z\overline{H_w}\,d\nu\\
&=\lambda\int_{[0,1)}K_{z,\mu}^\ast\overline{K_{w,\mu}^\ast}\,d\mu+(1-\lambda)\int_{[0,1)}K_{z,\nu}^\ast\overline{K_{w,\nu}^\ast}\,d\nu\\
&=K_{z}(w).\end{align*}
Consequently, $K_z$ reproduces itself with respect to its $L^2(\rho)$-boundary, and so $\rho\in\mathcal{M}(K)$.\end{proof}

For a nonconstant inner function $b$, let $\mu_n$ denote the unique singular measure on $[0,1)$ corresponding to $b^n$ via the Poisson integral.  Note that $\mathcal{H}(b) \subset \mathcal{H}(b^{n})$.

\begin{proposition}If $K$ is a positive matrix in $H^2$ such that $\mu=\mu_1\in\mathcal{M}(K)$ and $K\subseteq\mathcal{H}(b)$, then $\mu_n\in\mathcal{M}(K)$ for all $n\geq1$.
\end{proposition}
\begin{proof}Let $n\in\mathbb{N}$. We have $\{K_z:z\in\mathbb{D}\}\subseteq\mathcal{H}(b)\subseteq\mathcal{H}(b^n)$, and since functions in $\mathcal{H}(b^n)$ have $L^2(\mu_n)$-boundaries, each $K_z$ has an $L^2(\mu_n)$-boundary $K_{z,\mu_n}^\ast$. Recall that the norms on $\mathcal{H}(b)$ and $\mathcal{H}(b^n)$ are both equal to the $H^2$ norm and hence equal to each other. We therefore have
\[ 
K_z(w) 
=\langle K_z,K_w\rangle_{\mathcal{H}(b)}
=\langle K_z,K_w\rangle_{\mathcal{H}(b^n)}
=\int_{0}^{1}K_{z,\mu_n}^\ast\overline{K_{w,\mu_n}^\ast}\,d\mu_n.
\] 
Thus $\mu_n\in\mathcal{M}(K)$.\end{proof}

Given that $\mathcal{H}(b)$ is so (relatively) well understood, it is a perhaps more interesting question to ask what happens when a positive matrix lies outside of $\mathcal{H}(b)$. Given a positive matrix $K_z(w)$ and an inner function $b$, for which $n$, if any, is $\mu_n\in\mathcal{M}(K)$? We propose to begin a study of this question here. We begin by revealing the relationship between $\mu$'s family of Clark measures and the measures $\mu_n$.

\begin{lemma}\label{expsumlemma}Let $b:\mathbb{D}\rightarrow\mathbb{D}$, and let $n\in\mathbb{N}$. Then for all $z\in\mathbb{D}$,
$$\frac{1}{n}\sum_{j=0}^{n-1}\frac{1+e^{-2\pi ij/n}b(z)}{1-e^{-2\pi ij/n}b(z)}=\frac{1+b^n(z)}{1-b^n(z)}.$$
\end{lemma}

\begin{proof}For $z\in\mathbb{D}$ such that $b(z)=0$, the equality is obvious. So suppose $z\in\mathbb{D}$ is such that $b(z)\neq0$. We have
\begin{align*}\sum_{j=0}^{n-1}\frac{1+e^{-2\pi ij/n}b(z)}{1-e^{-2\pi ij/n}b(z)}&=\sum_{j=0}^{n-1}\frac{e^{2\pi ij/n}+b(z)}{e^{2\pi ij/n}-b(z)}\\
&=\sum_{j=0}^{n-1}\frac{e^{2\pi ij/n}}{e^{2\pi ij/n}-b(z)}+\sum_{j=0}^{n-1}\frac{b(z)}{e^{2\pi ij/n}-b(z)}\\
&=\sum_{j=0}^{n-1}\frac{1}{1-e^{-2\pi ij/n}b(z)}-\sum_{j=0}^{n-1}\frac{1}{1-\frac{e^{2\pi ij/n}}{b(z)}}.
\end{align*}

Observe that
\begin{align*}\sum_{j=0}^{n-1}\frac{1}{1-e^{-2\pi ij/n}b(z)}&=\sum_{j=0}^{n-1}\sum_{l=0}^{\infty}(e^{-2\pi ij/n}b(z))^l\\
&=\sum_{l=0}^{\infty}b^l(z)\begin{cases}0&\text{if }l\neq 0\text{ mod } n\\n&\text{if }l=0\text{ mod } n\end{cases}\\
&=\frac{n}{1-b^n(z)}.
\end{align*}
A similar computation shows that
$$\sum_{j=0}^{n-1}\frac{1}{1-\frac{e^{2\pi ij/n}}{b(z)}}=\frac{n}{1-\frac{1}{b^n(z)}}.$$
Hence,
\[
\frac{1}{n}\sum_{j=0}^{n-1}\frac{1+e^{-2\pi ij/n}b(z)}{1-e^{-2\pi ij/n}b(z)}=\frac{1}{1-b^n(z)}-\frac{1}{1-\frac{1}{b^n(z)}} 
=\frac{1+b^n(z)}{1-b^n(z)}.
\]
\end{proof}

\begin{lemma} \label{expmulemma} Given an inner function $b$, if $\mu_n$ is the singular measure associated to $b^n$, then we have
$$\mu_n=\frac{1}{n}\sum_{j=0}^{n-1}\sigma_{e^{2\pi ij/n}},$$
where $\sigma_{\alpha}$ is the singular measure corresponding to the inner function $\overline{\alpha}b$.
\end{lemma}

\begin{proof}By Lemma \ref{expsumlemma}, we have
\begin{align*}\text{Re}\left(\frac{1+b^n(z)}{1-b^n(z)}\right)&=\text{Re}\left(\frac{1}{n}\sum_{j=0}^{n-1}\frac{1+e^{-2\pi ij/n}b(z)}{1-e^{-2\pi ij/n}b(z)}\right)\\
&=\frac{1}{n}\sum_{j=0}^{n-1}\text{Re}\left(\frac{1+e^{-2\pi ij/n}b(z)}{1-e^{-2\pi ij/n}b(z)}\right)\\
&=\frac{1}{n}\sum_{j=0}^{n}\int_{0}^{1}\frac{1+\lvert z\rvert^2}{\lvert z-\xi\rvert^2}\,d\sigma_{e^{2\pi ij/n}}(\xi)\\
&=\int_{0}^{1}\frac{1+\lvert z\rvert^2}{\lvert z-\xi\rvert^2}\,d\left[\frac{1}{n}\sum_{j=0}^{n-1}\sigma_{e^{2\pi ij/n}}\right](\xi)
\end{align*}
This shows that $\frac{1}{n}\sum_{j=0}^{n-1}\sigma_{e^{2\pi ij/n}}$ is the singular measure corresponding to the inner function $b^n$ via the Herglotz representation theorem, which completes the proof.
\end{proof}

\begin{theorem}Let $K_z(w)$ be a positive matrix and let $b$ be an inner function. Let $m$, $n$, and $q$ be positive integers such that $n=qm$. Let $$\rho=\frac{q}{(q-1)n}\sum_{\substack{j=0\\q\nmid j}}^{n-1}\sigma_{e^{2 \pi i j/m}.}$$ If two of the measures $\mu_m$, $\mu_n$, and $\rho$ are in $\mathcal{M}(K)$, then so is the third.
\end{theorem}

\begin{proof}By Lemma \ref{expmulemma}, we have
\begin{align*}\mu_n&= 
\frac{1}{n}\left(\sum_{\substack{j=0\\q\mid j}}^{n-1}\sigma_{e^{2 \pi i j/(qm)}}+\sum_{\substack{j=0\\q\nmid j}}^{n-1}\sigma_{e^{2 \pi i j/n}}\right)\\
&=\frac{1}{n}\left(\sum_{j=0}^{m-1}\sigma_{e^{2 \pi i j/m}}+\sum_{\substack{j=0\\q\nmid j}}^{n-1}\sigma_{e^{2 \pi i j/n}}\right)\\
&=\frac{1}{q}\mu_m+\frac{q-1}{q}\rho.
\end{align*}
So, each of the measures $\mu_n$, $\mu_m$, and $\rho$ is in the affine hull of the other two. 

Recall that the Clark measures $\{\sigma_\alpha:\alpha\in\mathbb{T}\}$ are mutually singular \cite{Pol93}. It follows that $\mu_m$ and $\rho$, since they are sums of Clark measures that do not share a common Clark measure, are mutually singular. Hence, if $\rho=\rho_a+\rho_s$ is the Lebesgue decomposition of $\rho$ with respect to $\mu_m$, we must have $\rho_a=0$, and hence $\frac{d\rho_a}{d\mu_m}=0$.

So the Radon-Nikodym derivative of the part of $\rho$ absolutely continuous to $\mu_m$ is bounded. Furthermore, it is clear that $\mu_m$ and $\rho$ are absolutely continuous with respect to $\mu_n$ with respective Radon-Nikodym derivatives $\frac{d\mu_m}{d\mu_n}\equiv\frac{1}{q}$ and $\frac{d\rho}{d\mu_n}\equiv\frac{q-1}{q}$. Therefore, by Lemma \ref{affinehulllemma}, if two of the three measures are in $\mathcal{M}(K)$, so is the third.
\end{proof}

\nocite{Nel57a,ARS09a}


\end{document}